\documentclass[12pt]{article}
\usepackage{graphicx}
\usepackage{amsthm, amsmath, amssymb}
\usepackage{fullpage}
\usepackage{enumitem}
\usepackage{lineno}
\usepackage[b]{esvect}
\usepackage{tikz}
\usetikzlibrary{shapes.multipart,shapes.geometric,topaths,calc}
\usepackage{hyperref}

\newtheorem{theorem}{Theorem}
\newtheorem{lem}[theorem]{Lemma}

\newtheorem*{observation*}{Observation}
\theoremstyle{remark}
\theoremstyle{definition}
\newtheorem{claim}{Claim}

\numberwithin{equation}{section}
\newcommand{\abs}[1]{\left\lvert#1\right\rvert}

\newcommand{\floor}[1]{\left\lfloor#1\right\rfloor}
\newcommand{\ceil}[1]{\left\lceil#1\right\rceil}

\DeclareMathOperator{\poly}{poly}
\DeclareMathOperator{\PHC}{poly_{HC}}

\DeclareMathOperator{\PF}{poly_{F_1}}

\DeclareMathOperator{\P2F}{poly_{F_2}}

\newcommand{\sH}{\ensuremath{\mathcal{H}}}
\newcommand{\rF}{{\rm F}}
\newcommand{\rHC}{{\rm {HC}}}

\begin{document}
%
 \title{Polychromatic colorings of complete graphs with respect to $1$-,  $2$-factors and Hamiltonian cycles}
 \author{
 Maria Axenovich\thanks{Karlsruhe Institute of Technology, Karlsruhe, Germany, \texttt{maria.aksenovich@kit.edu}.} 
 \and John Goldwasser\thanks{ West Virginia University, Morgantown, WV, USA, \texttt{jgoldwas@math.wvu.edu}.}
 \and Ryan Hansen\thanks{ West Virginia University, Morgantown, WV, USA, \texttt{rhansen@mail.wvu.edu}}
 \and Bernard Lidick\'y\thanks{Iowa State University, Ames, IA, USA, \texttt{lidicky@iastate,edu}. Supported by NSF grant DMS-1600390.}
 \and Ryan R. Martin \thanks{Iowa State University, Ames, IA, USA, \texttt{rymartin@iastate.edu}. Research supported in part by Simons Foundation Collaboration Grant (\#353292, to R.R. Martin).}
 \and David Offner\thanks{Westminster College, New Wilmington, PA, USA, \texttt{offnerde@westminster.edu}.}
 \and John Talbot\thanks{University College London, London, UK, \texttt{j.talbot@ucl.ac.uk}.}
 \and Michael Young\thanks{Iowa State University, Ames, IA, USA, \texttt{myoung@iastate.edu}.}
 }
\date{\today}

\maketitle

\begin{abstract}
If $G$ is a graph and $\mathcal{H}$ is a set of subgraphs of $G$, then  an edge-coloring of $G$ is called   $\mathcal{H}$-polychromatic if every graph from $\mathcal{H}$ gets all colors present in $G$ on its edges. The $\mathcal{H}$-polychromatic number of $G$, denoted $\poly_\mathcal{H} (G)$, is the largest number of colors in an $\sH$-polychromatic coloring. In this paper, $\poly_\mathcal{H} (G)$ is determined exactly when $G$ is a complete graph and $\mathcal{H}$ is the family of all $1$-factors.
In addition $\text{poly}_\mathcal{H}(G)$ is found up to an additive constant term  when $G$ is a complete graph and $\mathcal{H}$ is the family of all $2$-factors, or the family of all Hamiltonian cycles.
\end{abstract}

\section{Introduction} 
\label{sec:introduction}
 
If $G$ is a graph and $\sH$ is a set of subgraphs of $G$, we say that an edge-coloring of $G$   is  $\sH$-{\it polychromatic} if  every graph from  $\sH$ gets all  colors present in $G$  on its edges.  The \emph{$\sH$-polychromatic number of $G$}, denoted $\poly_\sH (G)$, is the largest number of colors in an $\sH$-polychromatic coloring.  
If an $\sH$-polychromatic coloring of $G$ uses $\poly_\sH (G)$ colors,  it is called an {\it optimal} $\sH$-polychromatic coloring of $G$.

\subsection{Background}

Let $Q_n$ denote the hypercube of dimension $n$. 
Let $G=Q_n$ and $\sH$ be the family of all subgraphs of $G$ isomorphic to  $Q_d$.
If $d$ is fixed and $n$ is large, then Alon, Krech, and Szab\'o~\cite{Alon:2007cd} showed that $\lfloor \frac{(d+1)^2}{4} \rfloor \le \poly_\mathcal{H}(Q_n) \le \binom{d+1}{2}$.  
Offner~\cite{Offner:2008vb} proved that the lower bound is tight for all sufficiently large values of $n$.
Bialostocki \cite{Bialostocki:1983wo} treated the special case when $d=2$ and $n \geq 2$. 
Goldwasser \emph{et al.}  \cite{group_paper}  considered the case where $\sH$ is  the  family of all subgraphs of $Q_n$ isomorphic to a $Q_d$  minus an edge or  a $Q_d$ minus a vertex.

If $T$ is a tree and $\sH$ is the set of all paths of length at least $r$, then $\poly_{\sH}(T) = \ceil{r/2}$, as was shown by Bollob\'as  \emph{et al.} \cite{BPRS}.
When $G=K_n$ and $\sH$ is the set of all $r$-vertex cliques, $\poly_{\sH}(G)$ was considered by Erd\H{o}s and Gy\'arf\'as \cite{EG,  G} with the respective colorings called balanced.
When $G$ is an arbitrary multigraph of minimum degree $d$, and $\sH$ is the set of all stars with center $v$ and leaves $N(v)$, $v\in V(G)$, then it was shown by Alon \emph{et al.} \cite{ABBBCSSZ}, that $\poly_{\sH} (G) \geq \lfloor (3d+1)/4 \rfloor.$  Goddard and Henning \cite{GH} considered vertex-colorings of graphs such that each open neighborhood contains a vertex of every color used in $G$.

Polychromatic colorings were also investigated for vertex-colored hypergraphs. These colorings are essential tools in studying covering problems which are of fundamental importance in general graph and hypergraph settings, especially in geometric hypergraphs,  and they exhibit connections to VC-dimension, see \cite{AKV,ABBBCSSZ, BPRS, PT}.
 
\subsection{Main Results}

In this paper, we consider the case where  $G$ is a complete graph and  $\sH$ is a family of spanning subgraphs. 
Let $\rF_1= \rF_1(n)$ be the family of all $1$-factors of $K_n$,  $\rF_2=\rF_2(n)$ be the family of all $2$-factors of $K_n$ and $\rHC=\rHC(n)$ be the family of all Hamiltonian cycles of $K_n$. 
Our main results are as follows:

\begin{theorem}\label{theorem:P1F}
	If $n$ is an even positive integer, then $\PF(K_n)=\floor{\log_2 n}$.
\end{theorem}

\begin{theorem}\label{theoremeasy}
        There exists a constant $c$ such that $\floor{ \log_2 2(n+1)} \leq \P2F(K_n)   \leq \PHC(K_n) \leq \floor{\log_2 n}+c$. Moreover, $\floor{ \log_2 \frac{8(n-1)}{3}} \leq \PHC(K_n)$.
\end{theorem}

It is claimed  in a follow-up paper~\cite{followup}, that in fact $\P2F(K_n)=\floor{\log_2 2(n+1)}$ and $\PHC(K_n)=\floor{ \log_2 \frac{8(n-1)}{3}}$ for $n \geq 3$.
However, the arguments there include more case analysis and greater detail than what is required for the small additive constant given in Theorem~\ref{theoremeasy}.

The paper is structured as follows. 
We start with basic definitions in Section~\ref{sec:def}.
In Section~\ref{Constructions}, we give constructions of polychromatic colorings, which provide the lower bounds
for Theorems~\ref{theorem:P1F} and \ref{theoremeasy}.
In Section~\ref{sec:T1}, we prove Theorem~\ref{theorem:P1F}.
Section~\ref{sec:T2} contains the proof of Theorem~\ref{theoremeasy}.

\section{Definitions}\label{sec:def}

Let the vertices of $K_n$ be denoted by $v_1,v_2,\ldots,v_n$. 
An edge-coloring  $\varphi$ is \emph{ordered at $v_i$} for $i \in [n]$ if there exists a color $a$, called the \emph{main color at $v_i$}, such that $\varphi(v_i v_j) = a$ for all $j \in \{i+1,\ldots,n\}$.
Notice that $v_{n-1}$ and $v_n$ are ordered with respect to any coloring. We define the main color of $v_n$ to be
the same as for $v_{n-1}$.
A vertex $v_i$ is \emph{unitary} if there are colors $a \neq b$ such that $v_i$ is incident with $n-2$ edges colored $a$ 
and one edge $v_iv_j$ colored $b$, where $v_j$ is unitary with $n-2$ incident edges colored $b$.
For $v_i$ unitary, we also call $a$ the \emph{main color}.

An edge-coloring is \emph{ordered} if all vertices are ordered
 with respect to some ordering of vertices.
See Figure~\ref{fig:F1} for an example of an ordered coloring.
We call an edge-coloring \emph{combed} if each vertex is  either ordered or unitary.
It is not difficult to show that if there is at least one unitary vertex in a combed coloring then either the first three vertices (and no others) are unitary with different main colors, as in Figure~\ref{fig:F2}, or the first four vertices (and no others) are unitary with two of them with one main color, and two with another.

Let $\varphi$ be an ordered or combed coloring.  The \textit{inherited coloring} is the vertex-coloring $\varphi'$ obtained by coloring each vertex with its main color.
Its \textit{inherited color class $M_i$ of color $i$} is the set of all vertices $v$ with $\varphi'(v) = i$.
 Let $M_t(j)= M_t \cap \{v_1, v_2, \ldots, v_j\}$.
In this paper, we shall always think of  the ordered vertices as arranged on a horizontal line with $v_i$ to the left of $v_j$ if $i<j$. 
We say that an edge $v_iv_j$, $i<j$,  goes from $v_i$ to the right and from $v_j$ to the left.  
If $\varphi$ is an edge-coloring of a graph $G$, the {\it maximum monochromatic degree}  of $G$ is the largest integer $d$ such that some vertex of $G$ is incident to $d$ edges of the same color.  
We say such a vertex is a \emph{max-vertex}. 
If $X$ is a subset of $V(K_n)$, we say  that the edge-coloring $\varphi$ of $K_n$  is 
	\begin{itemize}
		\item {\it $X$-constant} if  for any $v\in X$,  $\varphi(v u)=\varphi(v w)$ for all  $u, w\in V\setminus X$,
		\item {\it $X$-ordered} if there is an ordering of the vertices such that $X = \{v_1,\ldots,v_m\}$ for some integer $m$ and $\varphi$ is ordered on vertices in $X$.
	\end{itemize}
Notice if a coloring $\varphi$ is $X$-ordered, it is also $X$-constant.

 \section{Constructions of Polychromatic Colorings}\label{Constructions}

We construct three edge-colorings of $K_n$, and show that they are polychromatic for $\rF_1$, $\rF_2$, and $\rHC$, respectively.

\subsection{$\rF_1$-polychromatic Coloring  $\varphi_{\rF_1}$} 
\label{subsec:_k_1f_polychromatic}

Let $n\geq 2$ be even, and let $k$ be the largest positive integer  such that $2^{k} \leq  n$, i.e., $k= \lfloor \log_2 n \rfloor$.  
Let $\varphi'$ be a vertex-coloring of $K_n$ with vertex set $\{v_1,\ldots,v_n\}$ and colors $1, \ldots, k$, where 
for each $i \in [k]$, $M_i$ is the color class of color $i$.
Moreover, for any $1\leq i<j\leq k$, every vertex in $M_i$ precedes every vertex in $M_j$,  and $ |M_t| =2^{t-1}$  for  $t=1, \ldots, k-1$. 
Hence the color  classes  $1, 2, \ldots, k$ have sizes $1, 2, 4, \ldots, 2^{k-2},n-2^{k-1}+1$, respectively.  Let $\varphi_{\rF_1}$ be the ordered coloring for which $\varphi'$ is the inherited coloring.

\tikzset{
vtx/.style={inner sep=1.7pt, outer sep=0pt, circle,fill=black},
edge1/.style={black,ultra thick,blue,line width=1.7pt},
edge2/.style={black,dotted,ultra thick,red},
edge3/.style={black,line width=0.7pt},
edge4/.style={black, densely dotted},
edge5/.style={black,dashed,thick},
edge1to/.style={edge1,-latex},
edge2to/.style={edge2,-latex},
edge3to/.style={edge3,-latex},
edge2match/.style={edge2,latex-latex},
edge3match/.style={edge3,latex-latex},
}

\begin{figure}
\begin{center}
\begin{tikzpicture}
\foreach \x in {1,2,...,10}{
\draw (\x,0) node[vtx,label=below:$v_{\x}$] {};
}
\draw (11,0) node{$\cdots$};
\draw (12,0) node{$\cdots$};
\draw (13,-0.2) node{$\cdots$};
\foreach \x/\y/\t in {1/1/1, 2/3/2, 4/7/3,  8/12/4}{
\draw (\x-0.3,-0.6) rectangle (\y+0.3,0.2)
(0.5*\x+0.5*\y, -1) node {$M_\t$}
;
\foreach \v in {\x,...,\y}{
\foreach \u in {1,...,5 }{
\draw[thick,edge\t] (\v,0) to[out=90,in=180] (0.5+\v+0.1*\u,\u*0.4);
}
}
}
\end{tikzpicture}
\end{center}
\caption{$\rF_1$-polychromatic coloring  $\varphi_{\rF_1}$.}\label{fig:F1}
\end{figure}
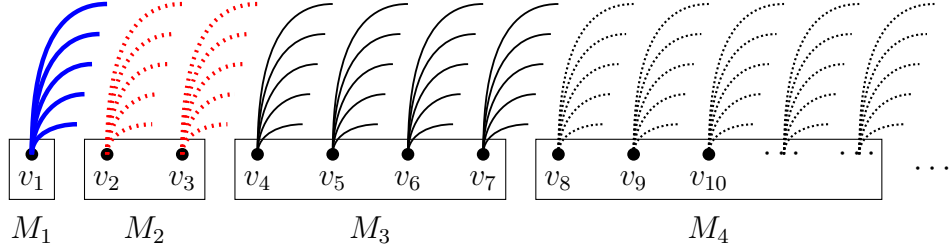

Consider an arbitrary $1$-factor $F$ of $K_n$ and $t\in [k]$.  Consider the set $F_t$ of all edges of $F$  with at least one endpoint in $M_t$.  Since  $|M_1|+\cdots+|M_{i}|= 2^{i} -1$ and    $|M_k| = n- |M_1|-\cdots - |M_{k-1}| \geq 2^{k}  -2^{k-1}+1= 2^{k-1}+1$, we have
 $\sum_{i< t}  |M_i| < |M_t|$ for any $t\in [k]$. Thus  at least one edge of $F_t$ joins a vertex from  $M_t$  to a vertex to the right, so this edge is of color $t$. Therefore $F$ has edges of each color.  Hence $\varphi_{\rF_1}$ is   $\rF_1$-polychromatic and it uses $\floor{\log_2 n}$ colors.

\subsection{$\rF_2$-polychromatic Coloring $\varphi_{\rF_2}$} 
\label{subsec:_k_2f_polychromatic_coloring}

Let $k$ be the largest positive integer such that $n\geq 2^{k-1} -1$, i.e., $k = 1+ \lfloor \log_2(n+1) \rfloor$.
Let $\varphi'$ be a vertex-coloring of $K_n$ with vertex set $\{v_1,\ldots,v_n\}$ and colors $1, \ldots, k$, where 
for each $i \in [k]$, $M_i$ is the color class of color $i$.
Moreover, for any $1\leq i<j\leq k$, every vertex in $M_i$ precedes every vertex in $M_j$,  and $ |M_t| =2^{t-2}$  for  $t=4, \ldots, k-1$, and  $|M_1|=|M_2|=|M_3|=1$.
Hence the color  classes  $1, 2, \ldots, k-1,k$ have sizes $1, 1, 1, 4, 8,  \ldots, 2^{k-3},n-2^{k-2}+1$, respectively.
 Let $\varphi_{\rF_2}$ be obtained by taking the ordered coloring for which $\varphi'$ is the inherited coloring
and then recoloring the edge $v_1v_3$ from color $1$ to color $3$.
 See Figure~\ref{fig:F2}.

\begin{figure}
\begin{center}
\begin{tikzpicture}
\foreach \x in {1,2,...,10}{
\draw (\x,0) node[vtx,label=below:$v_{\x}$] {};
}
\draw (11,0) node{$\cdots$};
\draw (12,0) node{$\cdots$};
\draw (13,-0.2) node{$\cdots$};
\foreach \x/\y/\t in {1/1/1, 2/2/2, 3/3/3,  4/7/4, 8/12/5}{
\draw (\x-0.3,-0.6) rectangle (\y+0.3,0.2)
(0.5*\x+0.5*\y, -1) node {$M_\t$}
;
\foreach \v in {\x,...,\y}{
\foreach \u in {1,...,5 }{
\draw[thick,edge\t] (\v,0) to[out=90,in=180] (0.5+\v+0.1*\u,\u*0.4);
}
}
}
\draw[thick,edge3] (1,0) to[out=100,in=100,looseness=4] (3,0);
\end{tikzpicture}
\end{center}
\caption{$\rF_2$-polychromatic coloring  $\varphi_{\rF_2}$.}\label{fig:F2}
\end{figure}
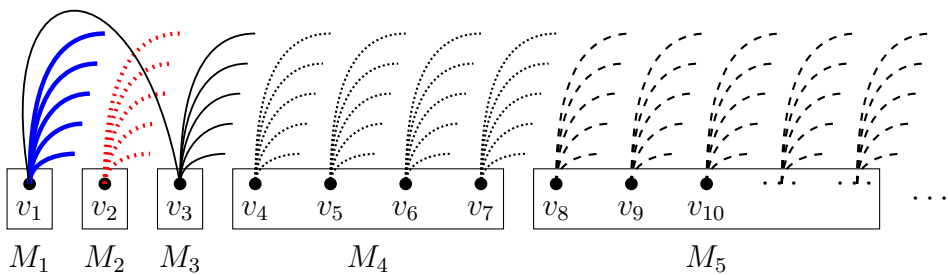

Observe that the inherited color classes $M_1$, $M_2$, and $M_3$ contain unitary vertices.  Moreover, $|M_1| + \cdots +|M_t| = 2^{t-1} -1$  for $3 \leq t \leq k-1$, 
and $|M_k| = n - |M_1|-\cdots -|M_{k-1}| \geq 2^{k-1} -1 - 2^{k-2} +1 = 2^{k-2}$. So, 
$\abs{M_t}>\sum_{i<t} \abs{M_i}$ for any $t \geq  4$.
Consider an arbitrary $2$-factor $F$ of $K_n$ and $t\in [k]$.
For $i \le 3$, $v_i$ is a unitary vertex with main color $i$, so $F$ must have edges of colors $1$, $2$, and $3$.
For a color $t\geq 4$, consider the set $F_t$ of edges of $F$ with endpoints in $M_t$. Then  $F_t$ has an edge of color $t$ unless $F_t$ forms a bipartite graph $G_t$ with one part $M_t$ and another $M_t'= \bigcup_{i=1}^{t-1} M_i$.  The degree of each vertex of $G_t$ from $M_t$ is two, and the degree of each vertex of $G_t$ from $M_t'$ is at most two. Thus $|M_t'|\geq |M_t|$, a contradiction.  Thus, $F_t$, and therefore $F$, has at least one edge of color $t$. So,  $\varphi_{\rF_2}$ is  $\rF_2$-polychromatic and it uses $k= \floor{\log_2 2(n+1)}$ colors.

\subsection{$\rHC$-polychromatic Coloring $\varphi_{\rHC}$} 
\label{subsec:_k_hc_polychromatic_coloring}

Let $k$ be the largest positive integer such that $n\geq 3\cdot 2^{k-3} +1$, i.e., $k = 3+ \lfloor \log_2 (n-1)/3 \rfloor$.
Let $\varphi'$ be a vertex-coloring of $K_n$ with vertex set $\{v_1,\ldots,v_n\}$ and colors $1, \ldots, k$, where 
for each $i \in [k]$, $M_i$ is the color class of color $i$.
Moreover, for any $1\leq i<j\leq k$, every vertex in $M_i$ precedes every vertex in $M_j$,  and $ |M_t| =3\cdot 2^{t-4}$  for  $t=4, \ldots, k-1$, and  $|M_1|=|M_2|=|M_3|=1$.
Hence the color  classes  $1, 2, \ldots, k-1,k$ have sizes $1, 1, 1, 3, 6, 12,   \ldots, 3\cdot 2^{k-5},n-3\cdot2^{k-4}$, respectively.
 Let $\varphi_{\rHC}$ be obtained by taking the ordered coloring for which $\varphi'$ is the inherited coloring
and then recoloring the edge $v_1v_3$ from color $1$ to color $3$.
 See Figure~\ref{fig:HC}.

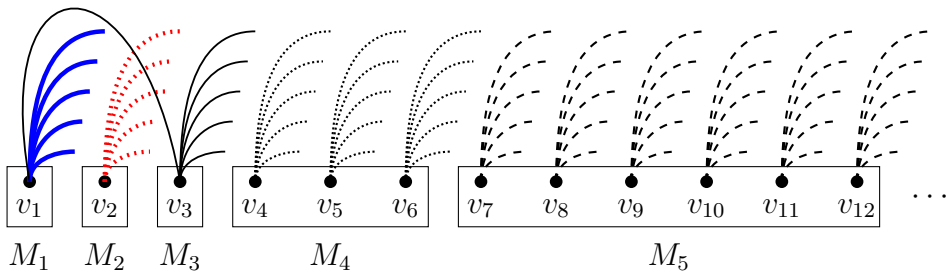
\begin{figure}
\begin{center}
\begin{tikzpicture}
\foreach \x in {1,2,...,12}{
\draw (\x,0) node[vtx,label=below:$v_{\x}$] {};
}
\draw (13,-0.2) node{$\cdots$};
\foreach \x/\y/\t in {1/1/1, 2/2/2, 3/3/3,  4/6/4, 7/12/5}{
\draw (\x-0.3,-0.6) rectangle (\y+0.3,0.2)
(0.5*\x+0.5*\y, -1) node {$M_\t$}
;
\foreach \v in {\x,...,\y}{
\foreach \u in {1,...,5 }{
\draw[thick,edge\t] (\v,0) to[out=90,in=180] (0.5+\v+0.1*\u,\u*0.4);
}
}
}
\draw[thick,edge3] (1,0) to[out=100,in=100,looseness=4] (3,0);
\end{tikzpicture}
\end{center}
\caption{$\rHC$-polychromatic coloring  $\varphi_{\rHC}$.}\label{fig:HC}
\end{figure}

We have that  $|M_1|+ \cdots +|M_t| = 3\cdot 2^{t-3}$ for $3\leq t\leq k-1$.
Moreover, $|M_k| = n- |M_1|-\cdots - |M_{k-1}| \geq 3\cdot 2^{k-3} +1 - 3\cdot 2^{k-4} = 3\cdot 2^{k-4} +1.$
Thus   $\abs{M_k}>\sum_{i<k} \abs{M_i}$  and $\abs{M_t}\geq \sum_{i<t} \abs{M_i}$ for all $t \geq 4$.
Consider an arbitrary Hamiltonian cycle $H$ of $K_n$. 
For $i \le 3$, $v_i$ is a unitary vertex with main color $i$, so $H$ must have edges of colors $1$, $2$, and $3$.
For each color  $t\geq 4$, let $H_t$ be the set of edges of $H$ with at least one endpoint in $M_t$.  
Then $H_t$ has an edge of color $t$ unless $H_t$ forms a bipartite graph $G_t$ with one part $M_t$ and another $M_t'= \bigcup_{i=1}^{t-1} M_i$.  The degree of each vertex of $G_t$ from $M_t$ is two, and the degree of each vertex of $G_t$ from $M_t'$ is at most two. If $4\leq t<k$,  $\abs{M_t}=\abs{M_t'}$, the degree of each vertex of $G_t$ from $M_t'$ is also two.  Hence $G_t$ is a union of cycles, so it could not be a proper subgraph of a Hamiltonian cycle.  If $t=k$, $|M_k|> |M_k'|$, so a bipartite graph $G_t$ could not exist. Thus $H$ has an edge of color $t$ for each $t=1, \ldots, k$,   $\varphi_{\rHC}$ is $\rHC$-polychromatic, and it uses $\floor{\log_2 \frac{8(n-1)}{3}}$ colors.

\section{Proof of Theorem \ref{theorem:P1F}}\label{sec:T1} 

We prove Theorem~\ref{theorem:P1F} by first showing the existence of an optimal edge-coloring
that is ordered. 
Then we use Lemma~\ref{orderedF1lemma} below which states that, for every inherited color class $M_t$,
there exists $j$ such that a majority of  $v_1,\ldots,v_j$ is in $M_t$.
This leads to a counting argument that  gives the upper bound in Theorem~\ref{theorem:P1F}.
For the lower bound we use the coloring $\varphi_{\rF_1}$.

\begin{lem}\label{orderedF1lemma}
	Let $\varphi:E(K_n) \to[k]$, where $n$ is even, be an ordered coloring with inherited color classes $M_1,\ldots,M_k$.
	If the coloring  $\varphi$ is $\rF_1$-polychromatic, then
	 $\forall t\in [k]$   $\exists  j\in[n-1]$ such that  $\abs{M_t(j)} > j/2$.
\end{lem}
\begin{proof}

Assume there exists $t$ such that for each $j\in[n-1]$, $\abs{M_t(j)} \leq j/2$.
Let $x_1, \ldots, x_m$ be the vertices of $M_t$ in order and let $y_1, \ldots, y_{n-m}$ be the other vertices of $K_n$ in order.
Let $H$ consist of the edges  $y_1x_1,  y_2x_2, \ldots, y_mx_m$ and a perfect matching on $\{y_{m+1}, \ldots, y_{n-m}\}$ (if this set is non-empty).
Since $|M_t(j)| \leq j/2$ for all $j$,  the number of $y$'s that must precede $x_i$ is at least $i$ for each  $i=1, \ldots, m$.
Hence $y_i$ is to the left of $x_i$ for each  $i=1, \ldots, m$.
Therefore all edges in $H$ incident with vertices in $M_t$ go to the left and do not have color $t$. 
The edges of $H$  that are not incident with vertices in $M_t$ are also not of color $t$.
Hence $\varphi$ is not $\rF_1$-polychromatic, a contradiction.
\end{proof}

\begin{proof}[Proof of Theorem \ref{theorem:P1F}]
Let $k=\PF(K_n)$ be the polychromatic number for  $1$-factors in $K_n=(V,E)$.  
Among all   $\rF_1$-polychromatic colorings of $K_n$ with $k$ colors we choose ones that are $X$-ordered for a subset  $X$ (possibly empty) of the largest size, and, of these, choose a coloring $\varphi$ whose restriction to $V\setminus X$ has the largest maximum monochromatic degree. 
Suppose for contradiction that $V \neq X$.

Let $Z = V\setminus X$ and $G$ be the subgraph of $K_n$ induced by $Z$.
Let $v$ be a vertex of maximum monochromatic degree, $d$, in $\varphi$ restricted to $G$, and let 1 be a color for which there are $d$ edges incident with $v$ in $G$ with color 1.
By the maximality of $\abs{X}$, there is a vertex $u$ in $Z$ such that $\varphi(uv)\neq 1$.  
Assume $\varphi(uv)=2$.  
If every $1$-factor containing $uv$ had another edge of color $2$, then the color of $uv$ could be changed to $1$, resulting in an $\rF_1$-polychromatic coloring where $v$ has a larger maximum monochromatic degree in $G$, a contradiction.  
Hence, there is a $1$-factor $F$ in which $uv$ is the only edge with color $2$ in $\varphi$.

Let  $\varphi(vy_i)=1$, $y_i\in Z$, $i=1, \ldots, d$.
For each $i\in[d]$, let $y_i w_i$ be the edge of $F$ containing $y_i$ (perhaps $w_i=y_j$ for some $j\neq i$);  see Figure~\ref{fig:1Fswitch}. 
We can get a different $1$-factor $F_i$ by replacing the edges $uv$ and $y_i w_i$ in  $F$ with edges  $v y_i$ and $u w_i$. 
Since $F_i$ must have an edge of color $2$ and $\varphi(v y_i)=1$, we must have $\varphi(u w_i)=2$ for each $i\in[d]$.

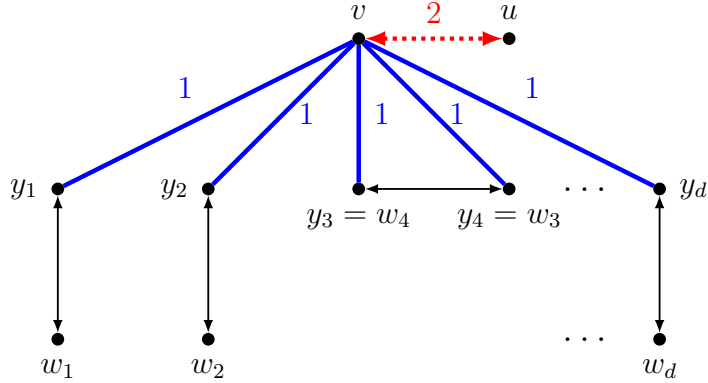
\begin{figure}[htbp]
	\centering
		\begin{tikzpicture}[every text node part/.style={align=center},scale=1,inner sep=1.75mm]
			\node[vtx,label=above:$v$] (v) at (5,4) {};
			\node[vtx,label=above:$u$] (u) at (7,4) {};
			\node[vtx,label=left:{$y_1$}] (y1) at (1,2) {};
			\node[vtx,label=left:{$y_2$}] (y2) at (3,2) {};
			\node[vtx,label=below:{$y_3=w_4$}] (y3) at (5,2) {};
			\node[vtx,label=below:{$y_4=w_3$}] (y4) at (7,2) {};
			\node (dots1) at (8,2) {\large $\ldots$};             
			\node[vtx,label=right:{$y_d$}] (yr) at (9,2) {};
			\node[vtx,label=below:$w_1$] (w1) at (1,0) {};
			\node[vtx,label=below:$w_2$] (w2) at (3,0) {};
			\node (dots2) at (8,0) {\large $\ldots$};
			\node[vtx,label=below:{$w_d$}] (wr) at (9,0) {};
			
			\draw[edge2match] (v) -- node[above] {2} (u);
			
			\draw[edge1] (v) -- node[above left] {1} (y1);
			\draw[edge1] (v) -- node[right] {1} (y2);
			\draw[edge1] (v) -- node[right] {1} (y3);
			\draw[edge1] (v) -- node[right] {1} (y4);
			\draw[edge1] (v) -- node[above right] {1} (yr);
			
			\draw[edge3match] (y3) -- (y4);
			\foreach \i in {1,2,r} {
				\draw[edge3match] (y\i) -- (w\i);
			}
		\end{tikzpicture}
	\caption{Maximum polychromatic degree in an $\rF_1$-polychromatic coloring.}
	\label{fig:1Fswitch}
\end{figure}

If $w_i\in X$ for some $i$ then, since $\varphi$ is $X$-constant, $\varphi(w_iy_i) = \varphi(w_iu) =2$,  so $y_i w_i$ and $uv$ are two edges of color $2$ in $F$, a contradiction. So, $w_i\in Z$ for all $i \in [d]$.
Thus $\varphi(u v)=\varphi(uw_1) = \cdots =\varphi(uw_d) = 2$, and the monochromatic degree of $u$ in $G$ is at least $d+1$, larger than that of   $v$, a contradiction.

We conclude that $X =V$. Hence $\varphi$ is an ordered  $\rF_1$-polychromatic coloring of $K_n$.
By Lemma \ref{orderedF1lemma}, for every $t \in [k]$ there exists $j_t$ such that $|M_t(j_t)| > \frac{j_t}{2}$.
By permuting the colors, we can assume $j_{t_1} < j_{t_2}$ whenever $t_1 < t_2$.
This gives us an ordering of  inherited color classes  $M_1,M_2,\ldots,M_k$. 
Since $|M_1|\ge 1$ and $|M_t(j_t)| > \frac{j_t}{2}$, we can use induction to show $|M_t| \geq |M_t(j_t)| \geq 2^{t-1}$
as follows
\[
|M_t| \geq |M_t(j_t)| > \sum_{1 \leq i < t} |M_i(j_t)| \geq \sum_{1 \leq i < t} |M_i(j_i)|  \geq  \sum_{1 \leq i < t} 2^{i-1} = 2^{t-1}-1.
\]
The sum of the sizes of all inherited color classes is $n$, and we get
\[
n = \sum_{t = 1}^{k}|M_t| \geq \sum_{t = 1}^{k} 2^{t-1} = 2^{k}-1.
\]
Since $n$ is even, $n \geq 2^k$ and $\PF(K_n)  = k\leq \floor{\log_2 n}$.

The fact that $\PF(K_n) \geq \floor{\log_2 n}$ follows from the coloring $\varphi_{\rF_1}$.
This finishes the proof of Theorem~\ref{theorem:P1F}.
\end{proof}

\section{Proof of Theorem \ref{theoremeasy}}\label{sec:T2} 
\label{Theorems}

Recall that we call an edge-coloring $\varphi$ \emph{combed} if all vertices are either ordered or unitary.

We prove Theorem~\ref{theoremeasy} by first showing the existence of an optimal edge-coloring
that is combed. Then we use Lemma~\ref{orderedPClemma} below which states that, for every inherited color class $M_t$,
either there exists $j$ such that at least half of  $v_1,\ldots,v_j$ is in $M_t$ or $M_t$ contains a unitary vertex.
This leads to a counting argument that finishes the proof of Theorem~\ref{theoremeasy}.

\begin{lem}\label{orderedPClemma}
	Let $\varphi:E(K_n) \to[k]$ be a combed coloring with inherited color classes $M_1,\ldots,M_k$.	 
	 If  the coloring $\varphi$ is $\rF_2$-polychromatic, or $\rHC$-polychromatic, then
	 $\forall t\in [k]$   $\exists  j\in[n-1]$ such that  $\abs{M_t(j)} \geq \frac{j}{2}$ or $M_t$ contains a unitary vertex.
\end{lem}

\begin{proof}
	Let $\sH \in \{\rF_2,\rHC\}$. 
	Let $\varphi$ be a combed $\sH$--polychromatic coloring with inherited color classes $M_1,\ldots,M_k$.
	It is sufficient to consider an arbitrary color $t \in [k]$ and show that the condition on $M_t$  is satisfied.
	
	Let $x_1, \ldots, x_m$ be the vertices of $M_t$ in order and let $y_1, \ldots, y_{n-m}$ be the other vertices of $K_n$ in order.
	Suppose for contradiction that there exists $t$ such that $\abs{M_t(j)} < \frac{j}{2}$ for all $j\in[n-1]$ and $M_t$ does not contain
	a unitary vertex.	
	Thus $\varphi$ is ordered at each $x_i \in M_t$ and so $y_{i+1}$ is to the left of  $x_i$ for each $i \in [m]$.
	Consider a Hamiltonian cycle $H=y_1 x_1 y_2 x_2 \cdots y_m x_m y_{m+1} \cdots y_{n-m} y_1$.  

Since $|M_t(j)| < j/2$ for all $j$,  the number of $y$'s that must precede $x_i$ is at least $i+1$ for each  $i=1, \ldots, m$.
Hence $y_i$ and $y_{i+1}$ are to the left of $x_i$ for each  $i=1, \ldots, m$.
Therefore each edge in $H$ incident with a vertex $x_i$ in $M_t$ goes to the left from the perspective of $x_i$.

Let $yx$ be an edge of $H$, where $x \in M_t$. 
Since $y \not\in M_t$, the majority color $r$ of $y$ is not $t$.
Since $\varphi$ is combed, either $\varphi(yx) = r$ or $\varphi(yx) \neq r$ and both $y$ and $x$ are unitary vertices.
Recall $M_t$ does not contain any unitary vertices. 
Hence no edge in $H$ is colored by $t$.
	This contradicts the fact that $\varphi$ is $\sH$-polychromatic.
\end{proof}

We say that a Hamiltonian cycle $H'$ is obtained  from a Hamiltonian cycle $H$  by a {\it twist} of disjoint  edges $e_1$ and $e_2$ of $H$ if $E(H)\setminus \{e_1, e_2\} \subseteq E(H')$,
i.e. we remove $e_1, e_2$ from $H$ and introduce two new edges to make the resulting graph a Hamiltonian cycle. Note that the choice of these two edges to add is unique.
The other choice of two edges to add does not preserve connectivity.
Without the connectivity requirement, the operation is known as a \emph{2-switch}.

Notice that any $2$-switch could be applied to a $2$-factor and the result will be again a $2$-factor.
Here, it might be possible to add the two new edges in two different ways.

For $\sH\in \{\rHC, \rF_2\}$, Lemma~\ref{structure} can be used to show that there exists an optimal $\sH$-polychromatic that is combed.

\begin{lem}\label{structure}
Suppose $n\geq 3$ and $X \subset V(K_n)$.
Let $\sH\in \{\rHC, \rF_2\}$ and $\varphi_1$ be an optimal $\sH$-polychromatic coloring of $K_n$ that is $X$-constant.
Then there exists an optimal $\sH$ polychromatic coloring $\varphi$ of $K_n$ that agrees with $\varphi_1$ on all edges with at least one endpoint in $X$ such that
\begin{enumerate}[label=(\Alph*)]
\item\label{caseA} there exists a vertex $v \in V(K_n)\setminus X$ such that $\varphi$ is $(X\cup\{v\})$-constant; or
\item\label{caseB} $X=\emptyset$ and there exist vertices $x,y,z$, such that 
these vertices are unitary under $\varphi$ of distinct main colors. 
This implies $\varphi$ is $\{x,y,z\}$-constant and $xyz$ is a rainbow triangle.
\end{enumerate}
\end{lem}

	\begin{proof}
		  Let $\sH\in\{\rF_2,\rHC\}$. 
		  Let $Z = V(K_n) \setminus X$ and $G$ be the subgraph of $K_n$ induced by $Z$.
		  Let $|Z| = m$. 
		  If $m\le 2$ then $X \neq \emptyset$ and \ref{caseA} is trivially satisfied.
		  Hence $m \geq 3$.
		  If $X=\emptyset$ and there exists an optimal $\sH$-polychromatic coloring $\varphi$ with three unitary vertices $x$, $y$, and $z$ of distinct main colors, then \ref{caseB} is satisfied. 
		  Hence we assume there is no such edge-coloring $\varphi$.
		  
		  Choose $\varphi$ to  be an optimal $\sH$-polychromatic coloring such that it agrees with $\varphi_1$ on edges with at least one endpoint in $X$ and subject to this, it maximizes the 
		  maximum monochromatic degree of $G$.
		  Define $d$ to be the maximum monochromatic degree of vertices in $G$ with $\varphi$.
		  
		    First suppose $d=m-1$. Let $v$ be a vertex of  maximum monochromatic degree $d$ in $G$. Then  $\varphi$ is  $(X\cup\{v\})$-constant and we have \ref{caseA}.
		    Hence we assume $d \leq m-2$.

Let $\ell= m-1-d$ and let $\varphi$ use colors $1, \ldots, k$.  
If color $a$ appears $d$ times in $G$ at a vertex $v \in Z$, we say $v$  is an \emph{$a$-max-vertex}.
If the $\ell$ edges incident with $v$ in $G$ which do not have color $a$ all have color $b$, we call $v$ an \emph{$(a,b)$-max-vertex} with \emph{minority color} $b$.

		\begin{claim}\label{claim1}
			If $a, b\in [k] $ are two distinct colors, $v \in Z$ is an $a$-max-vertex and $\varphi(vu)=b$ for some other vertex $u \in Z$, then all of the following hold:
\begin{enumerate}[label=($\arabic*$)]
\item\label{c1.1}  $u$ is a $b$-max-vertex,
\item\label{c1.2} $v$ is an $(a,b)$-max-vertex,
\item\label{c1.3}  either $X = \emptyset$ or $\sH = \rF_2$, and
\item\label{c1.4}  at least half of the edges between $X$ and $Z$ have color $b$.
\end{enumerate}
\end{claim}
\begin{proof}
	For ease of notation, we assume that $a=1$ and $b=2$.  Let $v$ be a $1$-max-vertex.
	Let $u \in Z$ be a vertex such that $\varphi(v u)=2$. 
	If every $H \in \sH$ containing $uv$ contains another edge of color $2$, we could change the color of $uv$ to color $1$, giving an $\sH$-polychromatic coloring where $v$ has monochromatic degree $d+1$, a contradiction. 
	 Hence, there must be  $H \in \sH$ where $uv$ is the only edge of color $2$.
	 
	 Cyclically orient the edges of each cycle in $H$ such that $uv$ is an arc, and denote the resulting directed graph $\vv{H}$.
	  Let $\varphi(vy_j)=1$,  for $y_j \in Z$,  $j=1,2,\ldots,d$.  For each $j\in[d]$, let $w_j$ be the predecessor of $y_j$ in $\vv{H}$, so $\vv{w_j y_j}\in \vv{H}$ for each $j$.  We assume $w_j\neq v$ for $j=2,3,\ldots,d$, but perhaps $w_1=v$ and perhaps $w_j=y_i$ for some $i\neq j$.
	  See Figure~\ref{fig:cl1}.

Now we shall prove \ref{c1.1}.		
If $w_j\neq v$, twist the edges $uv$ and $w_j y_j$ of $H$ to get a new $H_j \in \sH$ containing $vy_j$ and $uw_j$.
Since $H_j$ must have an edge of color $2$ and $\varphi(vy_j)=1$,  we must have  $\varphi(u w_j)=2$.  
Hence, $\varphi(uv)=\varphi(uw_2)=\varphi(uw_3)=\cdots=\varphi(uw_d)=2$.
Note that $w_j\in Z$ for each $j\in[d]$. 
This is because if $w_j\in X$, then, since $\varphi$ is $X$-constant and $y_j\in Z$, $\varphi(w_j y_j)=\varphi(w_j u)=2$, so $w_j y_j$ is another edge in $H$ with color $2$, a contradiction.  
That gives us $d$ edges of color $2$ at $u$ in $G$. 
Note that if $w_1\neq v$, then $uw_1$ is another edge of color $2$ incident to $u$, so $w_1=v$ and $\vv{vy_1}$ is an arc of $\vv{H}$.
Therefore, $u$ is a $2$-max-vertex. This proves \ref{c1.1}.

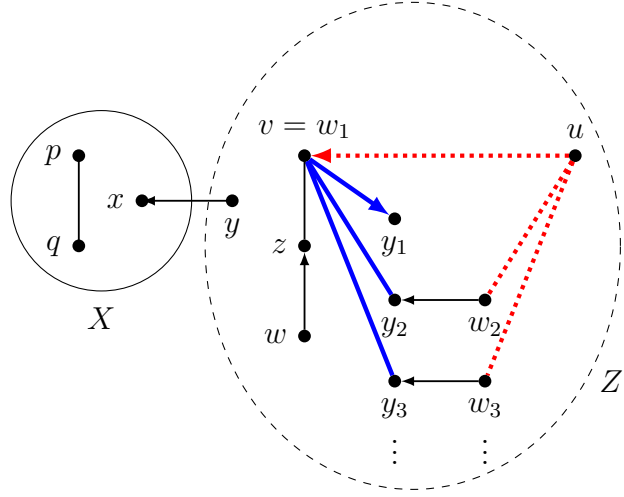
\begin{figure}
\begin{center}
\begin{tikzpicture}[scale=1.2]
\draw(-0.25,0) ellipse (1 and 1)
(-0.25,-1.3) node {$X$};
\draw[edge3to]
(1.2,0) node[vtx,label=below:$y$]{} -- ++(-1,0) node[vtx,label=left:$x$]{}
;
\draw[edge3]
(-0.5,-0.5) node[vtx,label=left:$q$]{} -- (-0.5,0.5) node[vtx,label=left:$p$]{}
;
\draw
(2,0.5) node[vtx,label=above:{$v=w_1$}](v){}
(5,0.5) node[vtx,label=above:{$u$}](u){}
\foreach \y in {1,2,3}{
(3,0.7-0.9*\y) node[vtx,label=below:{$y_\y$}](y\y){}
}
\foreach \w in {2,3}{
(4,0.7-0.9*\w) node[vtx,label=below:{$w_\w$}](w\w){}
}
(3,-2.7) node {$\vdots$}
(4,-2.7) node {$\vdots$}
(v) ++(0,-1) node[vtx,label=left:{$z$}](z){}
     ++(0,-1) node[vtx,label=left:{$w$}](w){}
;
\draw[edge2to] (u)--(v);
\draw[edge1to] (v)--(y1);
\draw[edge1] (v)--(y2) (v)--(y3);
\draw[edge2] (u)--(w2) (u)--(w3);
\draw[edge3to] (w2)--(y2);
\draw[edge3to] (w3)--(y3);
\draw[edge3to] (w)--(z);
\draw[edge3] (z)--(v);
\draw[dashed]
(3.2,-0.5) ellipse (2.3cm and 2.7cm)
(5.4,-2) node{$Z$}
;
\end{tikzpicture}
\end{center}
\caption{Situation in Claim~\ref{claim1}.}\label{fig:cl1}
\end{figure}

Next, we prove \ref{c1.2}, i.e., that $v$ is a $(1,2)$-max-vertex.
Let $z \in Z$ be a vertex distinct from $u$ such that $\varphi(vz)\neq 1$.  
Let $w$ be the vertex such that $\vv{wz}$ is an arc in $\vv{H}$. 
We know that $w\not\in\{v=w_1,w_2,\ldots,w_d, u\}$, since $z \not\in \{y_1,\ldots,y_d, v\}$.
Let $H_z \in \sH$, containing $vz$ and $uw$, be obtained from $H$ by twisting  $uv$ and $wz$.
Since $uv$ was the unique edge of $H$ colored by $2$, either $\varphi(u w)= 2$ or $\varphi(v z)=2$.
Suppose $w \in Z$. Since the maximum monochromatic degree is $d$ and  $\varphi(u w_j)=2$ for all $j\in[d]$, $\varphi(u w)\neq 2$, so $\varphi(v z)=2$.
Suppose $w \in X$. Since $\varphi(w z) \neq 2$ and $\varphi$ is $X$-constant, $\varphi(w z)=\varphi(u w)\neq 2$, so $\varphi(v z)=2$.
In both cases,  $\varphi(v z)=2$.
Therefore, $v$ is a $(1,2)$-max-vertex and the proof of \ref{c1.2} is done.

If $X=\emptyset$ then both \ref{c1.3} and \ref{c1.4} hold. So, assume that $X\neq \emptyset$.
Let $H\in \sH$.
Assume that there in an edge of $H$ with one endpoint in $X$ and another in $Z$. Then there exist $x \in X$ and $y \in Z$ such that $\vv{yx}$ is an arc in $\vv{H}$.
We know $y\not\in\{v=w_1,\ldots,w_d,u\}$, because the successor of $y$ in $\vv{H}$ is in $X$.
If we twist $yx$ and $uv$ we get $H_x \in \sH$
containing $uy$ and $vx$, where one of these edges must have color $2$.
However, since $\varphi(xv)=\varphi(xy)\neq 2$, we must have $\varphi(yu)=2$, and $u$ has monochromatic degree $d+1$ in $G$, a contradiction. Hence there is no edge in $\vv{H}$ with one endpoint in $X$ and another in $Z$, and thus $X$ induces a $2$-factor in $H$. In particular, since $Z\neq \emptyset$, $H$ is not a Hamiltonian cycle, and $\sH=\rF_2$.
Let $p,q\in X$ with $pq \in E(H)$.
Since both $(H\setminus\{uv,pq\})\cup\{pv,qu\}$ and $(H\setminus\{uv,pq\})\cup\{pu,qv\}$ are $2$-factors in $\mathcal{H}$, and $\varphi$ is $X$-constant, either $\varphi(pv)=\varphi(pu)=2$ or $\varphi(qv)=\varphi(qu)=2$.
In fact, since $\varphi$ is $X$-constant, for each edge $pq$ of $H$, where $p,q\in X$, all the edges from either $p$ or $q$ into $Z$ have color $2$. Since $H[X]$ is a union of cycles, at least half the edges between $X$ and $Z$ have color $2$.
This proves~\ref{c1.3} and \ref{c1.4} and finishes the proof of Claim~\ref{claim1}.
\end{proof}

\begin{claim}\label{claim123}
The graph $G$ does not contain a $(1,2)$-max-vertex,  a $(2,3)$-max-vertex, and a $(3,1)$-max-vertex at the same time.
\end{claim}
\begin{proof}
Let $x$, $y$, and $z$ be a $(1,2)$-max-vertex, a $(2,3)$-max-vertex, and a $(3,1)$-max-vertex, respectively.
Applying Claim~\ref{claim1} to $\{v,u\} = \{x,y\}$, then $\{v,u\}= \{y,z\}$, and then with $\{v,u\}=\{z,x\}$, the conclusion~\ref{c1.4} gives that at least half of the edges between $X$ and $Z$ have color $2$, $3$, and $1$, respectively.
Since colors $1$, $2$, and $3$ are distinct, we conclude $X = \emptyset$.
Let $H \in \sH$. 
Observe  that $x$, $y$, and $z$ could be incident only with edges of $H$ with colors in $\{1,2,3\}$ in $\varphi$, so all other colors in $H$ come from edges not incident with these vertices.

Let $\varphi^\star$ be obtained from $\varphi$ by the following modification
\[
c^\star(uv) = \begin{cases}
1 & u = x \text{ and } v \neq y,\\
2 & u = y \text{ and } v \neq z,\\
3 & u = z \text{ and } v \neq x,\\
\varphi(uv) & \text{otherwise}.
\end{cases}
\]
Observe that the union of edges of $H$ with at least one endpoint in $\{x,y,z\}$
contains all colors $\{1,2,3\}$ in $\varphi^\star$. 
Hence $H$ is polychromatic in $\varphi^\star$ and $\varphi^\star$ is $\sH$-polychromatic.
Moreover, $\varphi^\star$ is $\{x,y,z\}$-constant and all the other properties of \ref{caseB} hold, which is a contradiction. 
This finishes the proof of Claim~\ref{claim123}.
\end{proof}

\begin{claim}\label{12to21}
If $v$ is a $(1,2)$-max-vertex and $u \in Z$ such that $\varphi(uv)=2$, then $u$ is a $(2,1)$-max-vertex.
\end{claim}
\begin{proof}
Let $v$ be a $(1,2)$-max-vertex and $u \in Z$ such that $\varphi(uv)=2$.
Claim~\ref{claim1} implies that $u$ is a $(2,\star)$-max-vertex.
Suppose for contradiction that $u$ is a $(2,3)$-max-vertex.
Since the number of edges  incident to $v$ colored $2$ 
is the same as  the number of edges  incident to $u$ colored $3$
and $\varphi(uv)=2$, there is a vertex $x$ such that $\varphi(ux) = 3$ and $\varphi(vx) = 1$.
Again by Claim~\ref{claim1}, $x$ is a $(3,1)$-max-vertex, contradicting Claim~\ref{claim123}.
\end{proof}

\begin{claim}\label{only12}
If there is a $(1,2)$-max-vertex, then there is no $(1,b)$-max-vertex for any $b \neq 2$.
\end{claim}
\begin{proof}
By symmetry suppose for contradiction that $v \in Z$ is a $(1,2)$-max-vertex and $u \in Z$ is a $(1,3)$-max-vertex.
Let $x\in Z$ be a vertex with $\varphi(vx) = 2$.
By Claim~\ref{12to21}, $x$ is a $(2,1)$-max-vertex.
Notice $\varphi(ux) \in \{1,2\} \cap \{1,3\} = \{ 1\}$.
Now Claim~\ref{12to21} applied to $x$ and $u$ gives that $u$ is a $(1,2)$-max-vertex, which is a contradiction.
\end{proof}

Claims~\ref{12to21} and \ref{only12} imply that if there is an $(a,b)$-max-vertex, then $\{a,b\}=\{1,2\}$.

Let $S$ be the set of all $(1,2)$-max-vertices, $T$ be the set of all $(2,1)$-max-vertices, and $W=Z\setminus(S\cup T)$.  
By Claim~\ref{12to21}, both $S$ and $T$ are not empty.
Edges within $S$ and from $S$ to $W$ must have  color $1$ (because any minority color edge at a max-vertex is incident to a max-vertex of that color), edges within $T$ and from $T$ to $W$ must have color $2$, and each edge between $S$ and $T$ must have color $1$ or $2$. 

Suppose $X = \emptyset$. Let $s \in S$ and $t \in T$.
Let $\varphi^\star$ be obtained from $\varphi$ by recoloring all edges incident to $s$ to $1$
and by recoloring all edges incident to $t$ and not incident to $s$ to $2$.
Notice that $\varphi^\star$ is $\sH$-polychromatic. 
This contradicts  the maximality of the monochromatic degree in $\varphi$.
Therefore, $X \neq \emptyset$.

By symmetry, we can assume $|S| \leq |T|$.
Let $v \in S$ and $u \in T$ be such that $\varphi(vu) = 2$.
By the maximality of the monochromatic degree of $v$ in $Z$, there exists $H \in  \sH$, where $uv$ is the unique edge colored by $2$.
Recall $\ell= m-1-d$.
Let $s_1,\ldots,s_\ell \in S$, where $\varphi(us_i ) = 1$ for all $ i \in [\ell]$.
Let  $\vv{H}$ be  directed cycle(s) obtained by orienting edges of $H$ such that $\vv{uv} \in \vv{H}$.
Let $t_i$ be a vertex such that $\vv{s_it_i} \in \vv{H}$ for all  $i \in [\ell]$.
See Figure~\ref{fig:12}.

\begin{figure}
\begin{center}
\begin{tikzpicture}[scale=1.2]
\draw(0,0) ellipse (0.4 and 1)
(0,0.5) node[vtx](a){}
(0,-0.5) node[vtx](b){}
(0,-1.3) node {$X$};
\foreach \z in {-20,0,20}{
  \draw[edge1] (a) -- ++ (\z:0.7);
  \draw[edge2] (b) -- ++ (\z:0.7);
}
\draw[edge1to]
(0.2,0) node[vtx,label=left:$x$]{} -- ++(1,0) node[vtx,label=below:$z$]{}
;

\begin{scope}[xshift=2cm]
\draw(0,0) ellipse (0.5 and 1.2)
(0.1,0.5) node[vtx,label=left:$v$](v){}
(0.1,0) node[vtx,label=left:$s_1$](s1){}
(0.1,-0.5) node[vtx,label=left:$s_2$](s2){}
(0,-1.5) node {$S$}
;
\draw[edge1]  (-0.2,-0.8) node[vtx]{} -- ++(0.4,0) node[vtx]{} ;
\end{scope}

\begin{scope}[xshift=4cm]
\draw(0,0) ellipse (0.5 and 1.2)
(-0.1,0.5) node[vtx,label=right:{$u=t_1$}](u){}
(-0.1,0) node[vtx,label=right:$t_2$](t2){}
(0,-1.5) node {$T$}
;
\draw[edge2]  (-0.2,-0.8) node[vtx]{} -- ++(0.4,0) node[vtx]{} ;
\end{scope}

\draw[edge2to] (u)--(v);
\draw[edge1to] (s1)--(u);
\draw[edge1to]  (s2)--(t2);
\draw[edge1] (u) -- (s2);
\draw[edge2] (v) -- (t2);

\begin{scope}[xshift=3cm,yshift=-2cm]
\draw(0,0) ellipse (0.4 and 0.4)
(0,0) node[vtx](w){}
(0.7,0) node {$W$}
;
\foreach \z in {0,10,20}{
  \draw[edge1] (w) -- ++(110+\z:0.7);
  \draw[edge2] (w) -- ++ (70-\z:0.7);
}
\end{scope}

\draw[dashed]
(3,-0.5) ellipse (2.3cm and 2.3cm)
(5.2,-2) node{$Z$}
;
\end{tikzpicture}
\end{center}
\caption{Final part of proof of Lemma~\ref{structure} with only $(1,2)$- and $(2,1)$-max-vertices, solid edges correspond to color $1$,  dotted edges correspond to color $2$.}\label{fig:12}
\end{figure}
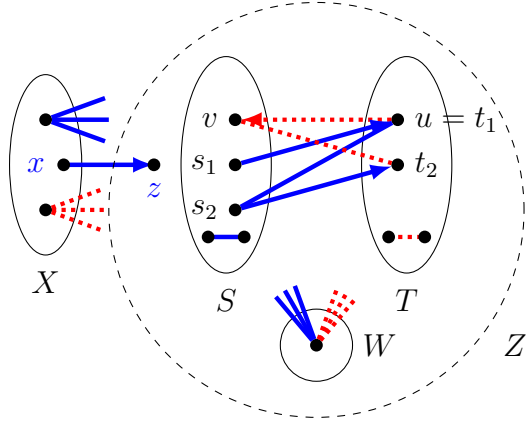

If there is  an $i \in [\ell]$ such that $\varphi(vt_i) \neq 2$, then twist of of $vu$ and $s_it_i$ 
contains $vt_i$ and $us_i$ and leaves no edge colored $2$, which is a contradiction with $\varphi$ being $\sH$-polychromatic. 
Hence $\varphi(vt_i) = 2$ and $t_i \in T$ for all $i \in [\ell]$.
Notice $v$ has exactly $\ell$ incident edges colored $2$ and the other ends of these edges must be  $t_1,\ldots,t_\ell$.
By symmetry, we assume $t_1 = u$.

Suppose there exists $\vv{xz}$ in $\vv{H}$ with $x \in X$ and $z \in Z$.
Since $uv$ is the unique edge of $\vv{H}$ colored $2$, $\varphi(xz)$ is not $2$ and
since $\varphi$ is $X$-constant, $\varphi(xz) = \varphi(xu) \neq 2$.
Notice that $z \not\in \{u=t_1,\ldots,t_\ell\}$ since for every $i \in [\ell]$, the predecessor 
of $t_i$ in $\vv{H}$ is $s_i$ and $s_i \in S$. 
Hence $\varphi(vz) = 1$ and the twist of $xz$ and $uv$ contains $xu$ and $vz$.
Since $\varphi(xu)\neq 2$ and $\varphi(vz)\neq 2$, we get a contradiction to $\varphi$ being $\sH$-polychromatic. 
Therefore, there is no edge of $H$ between $X$ and $Z$. 

Since there is no edge of $H$ between $X$ and $Z$ and $X \neq \emptyset$, $H$ is not connected.
Therefore, $\sH = \rF_2$.

Recall that all edges between $T$ and $W$ have color $2$, hence they are not in $H$.
Since there are no edges of $H$ between $X$ and $Z$, and all  edges within $T$ have color $2$,  every vertex in $T$ has both neighbors from $H$ in $S$. 
On the other hand, every vertex in $S$ has at most two neighbors from $H$ in $T$. Thus $|S|\geq |T|$.
Recall we assumed $|S| \leq |T|$.  
Hence $|S| = |T|$ and there are no edges of $H$ between $S \cup T$ and $W$.

Consider a bipartite graph $B$ with vertex set $S \cup T$, edges $st$, $s \in S$, $t\in T$, and $\varphi(st) = 1$.
Since vertices in $T$ are $(2,1)$-max-vertices, each of them has degree exactly $\ell$ in $B$.
Similarly, since vertices in $S$ are $(1,2)$-max-vertices, each of them is not adjacent to exactly
$\ell$ vertices of $T$. Therefore, all vertices in $S$ have the same degree in $B$.
Since $|S| = |T|$, we conclude $B$ is an $\ell$-regular graph.

If $\ell \geq 2$ then there exists a $2$-factor $K$ in $B$.
Let $H^\star$ be obtained from $H$ by removing edges incident to vertices in $S \cup T$ and
adding $K$. 
Since all edges of $K$ have color $1$ and $uv$ was the unique edge of $H$ colored $2$,
we conclude $H^\star$ has no edge colored $2$, contradicting the assumption that $\varphi$ is $\sH$-polychromatic.

Finally, if $\ell = 1$, then $B$ is a matching on $4$ vertices and the other two edges
between $S$ and $T$ must have color $2$.
Hence $S \cup T$ does not contain a $2$-factor in which $uv$ would be the unique edge colored $2$.
This contradicts the existence of $H$.

This finishes the proof of Lemma~\ref{structure}.		
\end{proof}

\begin{proof}[Proof of Theorem \ref{theoremeasy}]
Let $\sH \in \{\rF_2,\rHC\}$. 
Let $\varphi_1$ be an optimal $\sH$-polychromatic coloring of $E(K_n)$ with $k$ colors and $[k]$ be the set of colors.
We choose $X= \emptyset$, then we repeatedly apply Lemma~\ref{structure}.  
In the first application, we may get Lemma~\ref{structure}(B) and get $X= \{x,y,z\}$ that are unitary of distinct colors or Lemma~\ref{structure}(A) and $|X|=1$. 
But after that Lemma~\ref{structure}(A) always applies.
Note that there are no unitary vertices except possibly $x,y$, and $z$ because each other vertex is incident to distinct colors $c_x, c_y, c_z$ that are main colors of $x,y$, and $z$.
This results in a combed edge-coloring $\varphi$ with zero or three first unitary vertices and all others being ordered vertices.

Let $M_i$ be the inherited color classes obtained from $\varphi$.
Let $M_1,\ldots,M_{k-3}$ be the  inherited color classes not containing $x, y$, or $z$.
By Lemma~\ref{orderedPClemma}, for each color class $M_t$ there is $j_t$ such that $|M_t(j_t)| \geq \frac{j_t}{2}$.
By symmetry, assume $j_i < j_t$ for all $1 \leq i < t \leq k-3$.
This leads to
\[
	\abs{M_t(j_t)}\geq \sum_{i<t}\abs{M_i(j_t)} \geq  \sum_{i<t}\abs{M_i(j_i)} 
\]
for $t=2,\ldots,k-3$ and $|M_1| \geq 1$.
Hence by induction we get
\[
\abs{M_t(j_t)}\geq  1 +  \sum_{2 \leq i < t}\abs{M_i(j_i)} \geq 1 +  \sum_{2 \leq i < t} 2^{i-2} = 2^{t-2}.
\]
Therefore, $|M_t| \geq 2^{t-2}$ for $t \geq 2$ and 
\[
n \geq \sum_{1 \leq t\leq k-3}\abs{M_t} \geq  1 + \sum_{2 \leq t\leq k-3} 2^{t-2}  \geq 2^{k-4}.
\]
Hence $k \leq \log_2n + 4$. 
By splitting the cases to (A) and (B), we could show $k \leq \log_2n + 2$.

The lower bounds in Theorem~\ref{theoremeasy} follow from colorings  $\varphi_{\rF_2}$ and
$\varphi_{\rHC}$.
Since every Hamiltonian cycle is also a $2$-factor, we obtain  $\P2F(K_n)   \leq \PHC(K_n)$.
\end{proof}

\section{Closing Remarks}
\label{sec:conjectures_and_closing_remarks}

We show above that $c$ from Theorem~\ref{theoremeasy} is at most 4.
It is possible to get a more precise version of Lemma~\ref{orderedPClemma}
and use it to get sharp bounds in Theorem~\ref{theoremeasy}. 
We do not provide the details in order to keep the paper short and less technical.
Details should be in the follow-up paper~\cite{followup} together with 
generalizations which allow $\sH$ to be the family of all $1$-regular or $2$-regular graphs that
span all but a fixed number of vertices.

\section{Acknowledgements}
This research was initiated at a workshop made possible by the Alliance for Building
Faculty Diversity in the Mathematical Sciences (DMS 0946431).

\bibliographystyle{abbrv}
\bibliography{Poly1HC2}
\end{document}